\documentclass[10pt]{article}
\usepackage{amsmath,amssymb,amsthm}
\usepackage{amsfonts, amsmath, amsthm, amssymb}
\usepackage{epsfig}
\usepackage{color}
\usepackage[normalem]{ulem}

\title{Magic rectangles with empty cells}
\author {
Abdollah Khodkar and David Leach\\
Department of Mathematics\\
University of West Georgia\\
Carrollton, GA 30118\\
{\tt akhodkar@westga.edu},
{\tt cleach@westga.edu}
}

\date{}

\newtheorem{prelem}{{\bf Theorem}}

 \newtheorem{theorem}{Theorem}
\newtheorem{corollary}[theorem]{Corollary}
\newtheorem{lemma}[theorem]{Lemma}
\newtheorem{remark}[theorem]{Remark}

\theoremstyle{definition}

\theoremstyle{remark}

\begin{document}

\maketitle

\begin{abstract}
A magic rectangle of order $m\times n$ with precisely $r$ filled cells in each row and precisely $s$ filled cells in each column, denoted $MR(m,n;r,s)$, is an arrangement of the numbers from 0 to $mr-1$ in an $m\times n$ rectangle such that each number occurs exactly once in the rectangle and the sum of the entries of each row is the same and the sum of the entries of each column is also the same. In this paper we study the existence of $MR(m,n;r,2)$, $MR(m,km;ks,s)$, and $MR(am,bm;bs,as)$. We also prove that there exists a nonconsecutive magic square set $NMSS(m, s; t)$ if and only if $m=s=t=1$ or $3\leq s\leq m$ and either $s$ is even or $mt$ is odd.
\end{abstract}

\section{Introduction}\label{SEC1}
A {\em magic rectangle} of order $m\times n$, $MR(m,n)$, is an arrangement of the numbers from 0 to $mn-1$ in an $m\times n$ rectangle such that each number occurs exactly once in the rectangle and the sum of the entries of each row is the same and the sum of entries of each column is also the same.
The following theorem is well known (see \cite {TH1, TH2} or \cite {sun}):

\begin{theorem}\label{TH:sun}
An $m \times n$ magic rectangle exists if and only if $m \equiv n \pmod 2$, $m + n > 5$, and $m, n > 1$.
\end{theorem}

A {\em magic rectangle set} $MRS(a, b; c)$ is a collection of $c$ rectangles of order $a \times b$ with numbers
$0,1, 2,\ldots, abc-1$, each appearing once, with all row sums in every rectangle equal to a constant $M$ and
all column sums in every rectangle equal to a constant $N$.

\begin{theorem}\label{magicrectangleset}
\cite{DF1, DF2, TRH} Let $a, b, c$ be positive integers such that $1 < a \leq b$. Then a magic
rectangle set $MRS(a, b; c)$ exists if and only if either $a, b, c$ are all odd or $a$ and $b$ are
both even
with $(a, b)\neq (2, 2)$.
\end{theorem}

We now define magic rectangles that contain some empty cells. A magic rectangle of order $m\times n$ with precisely $r$ filled cells in each row and precisely $s$ filled cells in each column, denoted $MR(m,n;r,s)$, is an arrangement of the numbers from 0 to $mr-1$ in an $m\times n$ rectangle such that each number occurs exactly once in the rectangle and the sum of the entries of each row is the same and the sum of entries of each column is also the same. By definition, $mr=ns$, $r\leq n$ and $s \leq m$. If $r=n$ or $s=m$, then the rectangle has no empty cell.

If $m=n$ in an $MR(m,n;r,s)$, then $r=s$ and we denote this square by $MS(m;s)$ and call it a {\em magic square with empty cells}.
A magic square of order $n$ is called $s$-{\em diagonal} if its entries all belong to $s$ consecutive diagonals.

In \cite {KL} the authors settle the existence of magic squares with empty cells.

\begin{theorem}\label{TH:KL} \cite {KL}
There is an $s$-diagonal magic square $MS(m;s)$ if and only if $m=s=1$ or $3\leq s\leq m$ and either $s$ is even or $m$ is odd.
\end{theorem}

A {\em nonconsecutive magic square} with $s$ filled cells in each row and in each column, denoted $NMS(m;s)$, has the same properties as an $MS(m;s)$ except that the $ms$ numbers in the square are not necessarily consecutive numbers.
A {\em nonconsecutive magic square set}, $NMSS(m, s; t)$, is a collection of $t$ nonconsecutive magic squares $NMS(m;s)$
with entries $0,1, 2,\ldots, mst-1$, each appearing once, with all row sums in every square and all column sums in every square equal to a constant $C$.

In the coming sections we investigate the existence of $MR(m,n;r,2)$, $MR(m,km;ks,s)$, and $MR(am,bm;bs,as)$.
We also prove that there exists a nonconsecutive magic square set $NMSS(m, s; t)$ if and only if $m=s=t=1$ or $3\leq s\leq m$ and either $s$ is even or $mt$ is odd.
For a given array $A$, by $(i,j;k)$ $\in A$ we mean that the cell $(i,j)$ of $A$ contains $k$.

\section{Magic rectangles with two cells filled in each column}\label{SEC2}
In this section we investigate the existence of an $MR(m,n;r, 2)$ for all $m,n,r$, where $2\leq m\leq n$ and
$3\leq r\leq n$. Note that if $m=2$ or $r=n$, then the rectangle cannot have empty cells. In addition, if $m=n$, then $r=2$ and there does not exist an $MR(m,m;2, 2)$.

If an $MR(m,n;r, 2)$ exists, we must have $mr=2n$. If $r$ is odd, say $r=2k+1$, then $n=m(2k+1)/2$, which implies that $m$ must be even. If an $MR(m,m(2k+1)/2; 2k+1,2)$ exists, then its row sums are
$$\dfrac{m(2k+1)(m(2k+1)-1)}{2m}=\dfrac{(2k+1)(m(2k+1)-1)}{2},$$
which is not an integer because $m$ is even. Hence, there does not exist an $MR(m,m(2k+1)/2; 2k+1,2)$.
If $r$ is even, say $r=2k$, then $n=km$. We now prove the main result of this section.

\begin{theorem}\label{Th.m.km.2k.2}
{\rm There exists an $MR(m,km;2k,2)$ for every $k\geq 2$ and $m\geq 2$.
}
\end{theorem}

\begin{proof}
Partition an empty rectangle $m\times km$ into $k$ empty squares $m\times m$,
say $S_{\ell}$, where $0\leq \ell\leq k-1$.
Note that the numbers in an $MR(m,km;2k,2)$ are $\{0, 1, 2, \ldots, km(2km-1)\}$. If $k$ is even,
We fill out two consecutive diagonals of each $S_{\ell}$ for $0\leq \ell\leq k-1$ with numbers of
$$A_{\ell}=\{\ell m, \ell m+1,\ldots, \ell m+m-1\} $$ and
$$B_{\ell}=\{(2k-\ell-1)m,(2k-\ell-1)m+1,\ldots, (2k-\ell)m-1\},$$
respectively. So each diagonal consists of $m$ consecutive numbers. If $k$ is odd, we fill out two consecutive diagonals of each $S_{\ell}$ for $0\leq \ell\leq k-2$ with numbers of
$A_{\ell}$ and $B_{\ell}$, respectively. The last two diagonals are filled with numbers in
$A_{k-1}\cup B_{k-1}$ and the numbers in each of these two diagonals are not consecutive.

It is easy to see that the set $\{A_{\ell},B_{\ell}\mid 0\leq \ell\leq k-1\}$ is a partition of the set $\{0, 1, 2, \ldots, km(2km-1)\}$. Hence, every necessary number occurs exactly once in the resulting
$m\times km$ rectangle.

We consider two cases.
\vspace{3mm}

\noindent {\bf Case 1: $k$ even}

\noindent For $\ell$ even and $0\leq \ell \leq k-2$

 $ \begin{array}{lll}
\mbox {Diagonal 1:}\quad (i,i;\ell m+i)& \mbox {for}& 0\leq i\leq m-1,\\
\mbox {Diagonal 2:}\quad (i+1,i;(2k-\ell)m-i-1 & \mbox {for}& 0\leq i\leq m-1.\\
\end{array}$

\noindent For $\ell$ odd and $1\leq \ell \leq k-1$

$ \begin{array}{lll}
\mbox {Diagonal 1:}\quad (i,i;(\ell+1) m-i-1)& \mbox {for}& 0\leq i\leq m-1,\\
\mbox {Diagonal 2:}\quad (i+1,i;(2k-\ell-1)m+i & \mbox {for}& 0\leq i\leq m-1.\\
\end{array}$\\
See Figure \ref {5.10.4.2}.
 Note that the addition in the first component of each triple $(i,j;k)$ is modulo $m$.

\begin{figure}[ht]
$$\begin{array}{|c|c|c|c|c|c|c|c|c|c|}\hline
0&&&&15&9&&&&14 \\ \hline
19&1&&&&10&8&&&\\ \hline			
&18&2&&&&11&7&& \\ \hline		
&&17&3&&&&12&6&  \\ \hline	
&&&16&4&&&&13&5  \\ \hline
\end{array}$$
 \caption{An $MR(5,10;4,2)$}
 \label{5.10.4.2}
\end{figure}

First we calculate the row sums. The row sum for row zero is:

$$\begin{array}{cl}
&\sum_{\ell=0, \ell \mbox{ even}}^{k-2} [\ell m+(2k-\ell)m-m]\\
+&\sum_{\ell=1, \ell \mbox{ odd}}^{k-1}[(\ell+1)m-1+(2k-\ell-1)m+m-1]\\
=&k^2m+k^2m-k=k(2km-1).
\end{array}$$

The row sum for $1\leq i\leq m-1$ is:
$$\begin{array}{cl}
&\sum_{\ell=0, \ell \mbox{ even}}^{k-2} [(\ell m+i)+ (2k-\ell)m-i]\\
+&\sum_{\ell=1, \ell \mbox{ odd}}^{k-1}[[(\ell+1)m+i-1+(2k-\ell-1)m+i-1]\\
=&k^2m+k^2m-k=k(2km-1).
\end{array}$$

Since there are two cells filled in each column, it is easy to see that the column sum for each column of $S_{\ell}$ is $2km-1$.
\vspace{3mm}

\noindent {\bf Case 2: $k$ odd}

\noindent For $0\leq \ell \leq (k-1)/2$

$ \begin{array}{lll}
\mbox {Diagonal 1:}\quad (i,i;\ell m+i)& \mbox {for}& 0\leq i\leq m-1,\\
\mbox {Diagonal 2:}\quad (i+1,i;(2k-\ell)m-i-1 & \mbox {for}& 0\leq i\leq m-1.\\
\end{array}$

\noindent For $(k+1)/2\leq \ell \leq k-2$

$ \begin{array}{lll}
\mbox {Diagonal 1:}\quad (i,i;(\ell+1) m-i-1)& \mbox {for}& 0\leq i\leq m-1,\\
\mbox {Diagonal 2:}\quad (i+1,i;(2k-\ell-1)m+i & \mbox {for}& 0\leq i\leq m-1.\\
\end{array}$

\noindent For $\ell = k-1$

$ \begin{array}{lll}
\mbox {Diagonal 1:}\quad (i,i;(k+1) m-2i-1)& \mbox {for}& 0\leq i\leq m-1,\\
\mbox {Diagonal 2:}\quad (i+1,i;(k-1)m+2i & \mbox {for}& 0\leq i\leq m-1.\\
\end{array}$\\
See Figure \ref{4.12.6.2}.
 Note that the addition in the first component of each triple $(i,j;k)$ is modulo $m$.

First we calculate the row sums. The row sum for row zero is:

$$\begin{array}{cl}
&\sum_{\ell=0}^{(k-1)/2} [\ell m+(2k-\ell)m-m]\\
+&\sum_{\ell=(k+1)/2}^{(k-2)} [(\ell+1)m-1 + (2k-\ell-1)m+m-1]\\
+&2km+2m-3\\
=& (2km-m)(k+1)/2 + (2km+m-2)(k-3)/2+2km+2m-3\\
=& k(2km-1).
\end{array}$$

The row sum for $1\leq i\leq m-1$ is:

$$\begin{array}{cl}
&\sum_{\ell=0}^{(k-1)/2} [(\ell m+i)+(2k-\ell)m-i]\\
+&\sum_{\ell=(k+1)/2}^{(k-2)} [(\ell+1)m-i-1 + (2k-\ell-1)m+i-1]\\
+&(k+1)m-2i-1)+(k-1)m+2(i-1)\\
=& (2km)(k+1)/2 + (2km-2)(k-3)/2+2km-3\\
=& k(2km-1).
\end{array}$$

Since there are two cells filled in each column, it is easy to see that the column sum for each column of
$S_{\ell}$ is $2km-1$.

\end{proof}

\begin{figure}[ht]
$$\begin{array}{|c|c|c|c|c|c|c|c|c|c|c|c|}\hline
 0&&&20	&4&&&16&15&&&14\\ \hline
23&1&&&19&5&&&8&13&&\\ \hline	
&22&2&&&18&6&&&10&11&	\\ \hline
&&21&3&&&17&7&&&12&9\\ \hline
\end{array}$$
 \caption{An $MR(4,12;6,2)$}
 \label{4.12.6.2}
\end{figure}


\section{The existence of an $MR(m,km;ks,s)$}\label{SEC3}
In this section we prove that there exists an $MR(m,km;ks, s)$, where
$k,m,s$ are positive integers, $3\leq s\leq m$, and either $m$ is even or $ks$ is odd.
An $m\times n$ array whose rows are permutations is called a {\em row-Latin} array. In the following
lemma we investigate the existence of row-Latin arrays whose column sum is constant. Such arrays
are called {\em Kotzig} arrays (see \cite{W}).

\begin{lemma}\label{rectangleofpermutations}
Let $k$ and $s$ be positive integers. If $s$ is even or $k$ is odd, then
there is an $s\times k$ rectangle whose rows are permutations on $0,1,2,\ldots, k-1$ and whose column sum is
$(k-1)s/2$ for each column.
\end{lemma}

\begin{proof}
First let $s=2$. Then the $2\times k$ rectangle $M$ with entries
$(0,j;j)$ and $(1,j; k-1-j)$, where $0\leq j\leq k-1$, is the required rectangle. If $s=2\ell$, one can use $\ell$ copies of $M$ to obtain an $s\times k$ rectangle with the required properties.

Now let $s$ and $k$ be both odd. We first find a
$3\times k$ rectangle with the required properties. Start with a $3\times k$ empty rectangle, say $N$, and fill its cells as follows.

\noindent \mbox {Row 0:}\quad $\begin{array}{lll}(0,j;j)& \mbox {for}& 0\leq j\leq k-1;\end{array}$

\noindent \mbox {Row 1:}\quad $\left\{\begin{array}{lll}(1,j;(k-1)/2+j)& \mbox {for}& 0\leq j\leq (k-1)/2,\\
                                       (1,j;j-(k+1)/2)& \mbox {for}& (k+1)/2\leq j\leq k-1;\\
                    \end{array}\right.$   \\
\noindent \mbox {Row 2:}\quad $\left\{\begin{array}{lll}(2,j;k-1-2j)& \mbox {for}& 0\leq j\leq (k-1)/2,\\
                                       (2,j;2k-1-2j)& \mbox {for}& (k+1)/2\leq j\leq k-1.\\
                    \end{array}\right.$

\noindent (See Figure \ref{P_Arrays}.)

\noindent It is easy to see that each row is a permutation of $0,1,2,\ldots, k-1$. The column sum for column $j$, where $0\leq j\leq (k-1)/2$, is
$$j+(k-1)/2+j+(k-1-2j)=3(k-1)/2$$
and the column sum for column $j$, where
$(k+1)/2\leq j\leq k-1$, is
$$j+(j-(k+1)/2)+(2k-1-2j)=3(k-1)/2,$$ as desired.

\noindent Now let $k\geq 5$ be odd. We use $(s-3)/2$ copies of $M$ and one copy of $N$ to obtain the required rectangle of order $s\times k$.  This completes the proof.
\end{proof}

\begin{figure}[ht]
$$\begin{array}{|c|c|c|c|c|c|c|c|c|}\hline
0&1&2&3&4&5&6&7&8 \\ \hline
4&5&6&7&8&0&1&2&3  \\ \hline
8&6&4&2&0&7&5&3&1  \\ \hline
\end{array}$$
\caption{An example of a $3\times 9$ Kotzig array}
 \label{P_Arrays}
\end{figure}

\begin{theorem}\label{Th.m.km.ks.s}
Let $k,m,s$ be positive integers. Then there exists a magic rectangle $MR(m,km;ks,s)$ if and only if
$m=s=k=1$ or $2\leq s\leq m$ and either $s$ is even or $km$ is odd.
\end{theorem}

\begin{proof}
If there exists an $MR(m,km;ks,s)$, then the row sum $(kms-1)kms/2m$ and column sum $(kms-1)kms/2km$
must be integers. This implies that $2\leq s\leq m$ and $s$ is even or $km$ is odd. (The case $m=s=k=1$ is trivial.)

Now we prove that these conditions are sufficient.
If $s=2$ we apply Theorem \ref{Th.m.km.2k.2}. Now let $3\leq s\leq m$.
By Theorem \ref{TH:KL}, there is an $s$-diagonal $m\times m$ magic square, say $A$.
Label the $s$ diagonals of $A$ by $0,1,2,\ldots, s-1.$
Partition an empty rectangle $m\times km$, say $B$, into $k$ empty $m\times m$ squares,
say $S_{\ell}$, where $0\leq \ell\leq k-1$. We fill $s$ diagonals of each $S_{\ell}$ as follows.
If $(a,b; c)\in A$, then $(a,b; c+\ell ms))\in S_{\ell}$, where $0\leq \ell\leq k-1$.
See Figure \ref{S_Arrays}.
Since the row sums of $A$ is $s(ms-1)/2$, it follows that the row sum for each row of $B$ is
$$\dfrac{k(ms-1)s}{2}+\dfrac{(k-1)kms^2}{2}=\dfrac{ks(kms-1)}{2},$$
which is also the row sum for each row of an $MR(m,km;ks,s)$ if it exists.
The column sum for each column of $S_{\ell}$ is $((ms-1)s)/2 + \ell ms^2$, where $0\leq \ell \leq k-1$.
We now permute the diagonals of $B$ such that the column sum for each column of $B$
becomes $s(kms-1)/2$. Note that permuting the diagonals of $B$ does not change the row sum for each row of $B$.
If $s$ is even or $k$ is odd, by Lemma
\ref{rectangleofpermutations}, there is an $s\times k$ rectangle, say $P$, whose rows are permutations on $0,1,2,\ldots, k-1$ and whose column sum is $s(k-1)/2$ for each column.
Partition an empty rectangle $m\times km$, say $C$, into $k$ empty squares $m\times m$, say $T_{\ell}$, where $0\leq \ell\leq k-1$.
For each $(i,j;\ell)\in P$, we place the entries of diagonal $i$ of $S_\ell$ in diagonal $i$ of $T_j$.
See Figure \ref{5.25.15.3}.
Since the rows of $P$ are permutations, we use diagonal $i$ of $S_{\ell}$ exactly once for $0\leq i\leq s-1$ and $0\leq \ell\leq k-1$. In addition, since the column sum of $P$ is $s(k-1)/2$ for each column,
it follows that the column sum for each column of $C$ is
$$\left(\dfrac{s(k-1)}{2}\right)ms+\dfrac{(ms-1)s}{2}=\dfrac{(kms-1)s}{2},$$
as desired.
\end{proof}

\begin{figure}[ht]
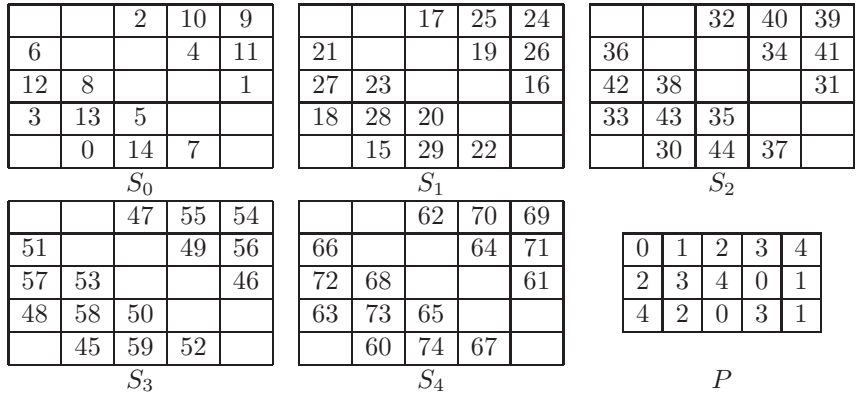

$$\begin{array}{ccc}
\begin{array}{|c|c|c|c|c|} \hline
&&2&10&9     \\ \hline
6&&&4&11     \\ \hline	
12&8&&&1    	\\ \hline
3&13&5&&    \\ \hline				
&0&14&7&   \\ \hline				
\end{array} &
\begin{array}{|c|c|c|c|c|} \hline
&&17&25&24      \\ \hline
21&&&19&26    \\ \hline	
27&23&&&16  	\\ \hline
18&28&20&&    \\ \hline				
 &15&29&22&    \\ \hline				
\end{array} &
\begin{array}{|c|c|c|c|c|} \hline
&&32&40&39 \\ \hline
36&&&34&41 \\ \hline	
42&38&&&31\\ \hline
33&43&35&& \\ \hline				
&30&44&37& \\ \hline				
\end{array} \\
S_0&S_1&S_2\\
\begin{array}{|c|c|c|c|c|} \hline
&&47&55&54     \\ \hline
51&&&49&56      \\ \hline
57&53&&&46     \\ \hline		
48&58&50&&       \\ \hline	
&45&59&52&     \\ \hline
\end{array}&
\begin{array}{|c|c|c|c|c|} \hline
&&62&70&69  \\ \hline
66&&&64&71  \\ \hline
72&68&&&61 \\ \hline		
63&73&65&&  \\ \hline	
&60&74&67&  \\ \hline
\end{array} &
\begin{array}{|c|c|c|c|c|}\hline
0&1&2&3&4  \\ \hline
2&3&4&0&1  \\ \hline
4&2&0&3&1   \\ \hline
\end{array}\\
S_3&S_4&P\\
\end{array}$$
\caption{Arrays $S_0, S_1, S_2, S_3, S_4$ and $P$}
 \label{S_Arrays}
\end{figure}

\begin{figure}[ht]
$$\begin{array}{ccc}
\begin{array}{|c|c|c|c|c|} \hline
&&62&40&9 \\ \hline			
6&&&64&41 \\ \hline	
42&8&&&61 \\ \hline	
63&43&5&&  \\ \hline
&60&44&7&  \\ \hline	
\end{array}&
\begin{array}{|c|c|c|c|c|} \hline
&&32&55&24	 \\ \hline		
21&&&34&56	\\ \hline
57&23&&&31	\\ \hline
33&58&20&&  \\ \hline			
&30&59&22&  \\ \hline					
\end{array}&
\begin{array}{|c|c|c|c|c|} \hline
&&2&70&39	\\ \hline		
36&&&4&71  \\ \hline
72&38&&&1 \\ \hline	
3&73&35&48& \\ \hline	   	
&0&74&37&  \\ \hline		    		
\end{array}\\
T_0&T_1&T_2 \\
\begin{array}{|c|c|c|c|c|} \hline
&&47&10&54  \\ \hline		
51&&&49&11	\\ \hline
12&53&&&46	\\ \hline
48&13&50&&  \\ \hline		
&45&14&52&  \\ \hline
\end{array}&
\begin{array}{|c|c|c|c|c|} \hline
&&17&25&69  \\ \hline
66&&&19&26  \\ \hline
27&68&&&16   \\ \hline
18&28&65&&  \\ \hline		
&15&29&67&  \\ \hline
\end{array}& \\
T_3&T_4&\\
\end{array}$$
\caption{An $MR(5,25;15,3)$}
 \label{5.25.15.3}
\end{figure}

Recall that a nonconsecutive magic square set $NMSS(m, s; t)$ is a collection of $t$ nonconsecutive magic squares $NMS(m;s)$ with entries $0,1, 2,\ldots, $ $mst-1$, each appearing once, with all row sums in every square and all column sums in every square equal to a constant.

\begin{corollary}\label{Cor.m.km.ks.s}
Let $m,s,t$ be positive integers.  Then there exists a nonconsecutive magic square set $NMSS(m, s; t)$ if and only if
$m=s=t=1$ or $3\leq s\leq m$ and either $s$ is even or $mt$ is odd.
In addition, the entries in each square are on $s$ consecutive diagonals.
\end{corollary}

\begin{proof}
Let the squares $S_{\ell}$ and $T_{\ell}$, $0\leq \ell\leq k-1$, be as in the proof of Theorem \ref{Th.m.km.ks.s}.
First note that if we add a constant to a diagonal of $S_0$, then the row sums and column sums within
the square $S_0$ will be equal to a constant $M$. This implies that the row sums and the column sums of every $S_{\ell}$, $0\leq \ell\leq k-1$, are equal (see Figure \ref{S_Arrays}). In addition, if we switch a diagonal of $S_i$ with its corresponding diagonal in $S_j$, the row sums and column sums of $S_i$ and of $S_j$ will be equal. This implies that the row sums and the columns sum of every $T_{\ell}$ given in the proof of Theorem \ref{Th.m.km.ks.s} are equal for $0\leq \ell\leq k-1$. By the proof of Theorem \ref{Th.m.km.ks.s}, the column sum of any column of $T_i$ is equal to the column sum of any column of $T_j$. Hence $T_{\ell}$, $0\leq \ell\leq k-1$, is a nonconsecutive magic square set $NMSS(m, s; t)$. By construction, the entries in each square are on $s$ consecutive diagonals.
By Theorem \ref{TH:KL}, one can see that the given conditions are necessary.
\end{proof}

\section{The existence of an $MR(am,bm;bs,as)$}\label{SEC4}
We first prove that if $m,n$ are positive integers with $\gcd(m,n)=1$, then there does not exist an $MR(m,n; r, s)$ with empty cells.

\begin{lemma}\label{mandnarecoprime}
Let $m,n,r,s$ be positive integers. If $\gcd(m,n)=1$, then there does not exist an $MR(m,n; r, s)$ with empty cells.
\end{lemma}

\begin{proof}
If an $MR(m,n; r, s)$ exists, then $mr=ns$. Since $\gcd(m,n)=1$, it follows that $n\mid r$. On the other hand, $r\leq n$. So $r=n$ and $m=s$. Hence the magic rectangle cannot have empty cells.
\end{proof}

Let $a$ and $b$ be positive integers with $\gcd(a,b)=1$. If an $MR(am,bm;$  $r, s)$ exists, then $amr=bms$ and by
the fact that $\gcd(a,b)=1$ we must have $r=br'$ and $s=as'$. Hence, by the equality $ambr'=bmas'$ we have $r'=s'$.
Now we investigate the existence of an
$MR(am,bm;bs,as)$.

\begin{theorem}\label{rectanglesmultiplication}
If there exist an $MS(m;s)$ and an $a\times b$ magic rectangle,
then there exists a magic rectangle $MR(am,bm;bs,as)$.
\end{theorem}

\begin{proof}
Let $A$ be an $MS(m;s)$ and let $B$ be an $a\times b$ magic rectangle. Partition an $am\times bm$ empty rectangle $C$
into $m^2$ rectangles of size $a\times b$, say $S_{(i,j)}$, where $0\leq i,j\leq m-1$. If $(i,j;k)\in A$ and
$(p,q,\ell)\in B$, then we fill the cell $(p,q)$ of $S_{i,j}$ with $kab+\ell$.
By construction, the entries in rectangle $C$ are $0,1,2,\ldots, abms-1$.
We now calculate row sums and column sums of $C$. The row sum for each row of $C$ is:
$$\frac{(ab-1)bs}{2}+\frac{(ms-1)s(ab)b}{2}=\frac{(abms-1)bs}{2},$$
as desired.

Similarly, the column sum for each column of $C$ is:
$$\frac{(ab-1)as}{2}+\frac{(ms-1)s(ab)a}{2}=\frac{(abms-1)as}{2},$$
as desired.
\end{proof}

\begin{corollary}\label{Cor:rectanglesmultiplication}
Let $3\leq s\leq m$ and either $s$ is even or $m$ is odd. Let $a \equiv b \pmod 2$, $a + b > 5$ and $a, b > 1$.
Then there exists an $MR(am,bm;bs,as)$.
\end{corollary}

\begin{proof}
By Theorem \ref{TH:KL}, there is an $s$-diagonal $MS(m;s)$ and by Theorem \ref{TH:sun}, there is an
$a\times b$ magic rectangle. Now the result follows by Theorem \ref{rectanglesmultiplication}.
\end{proof}


\begin{remark}\label{remark}
{\em
A closer look at the direct constructions given in Theorems 3, 4, 5, 6 and the inductive constructions given in Section 3 of \cite{KL} yields the following facts, which we employ in the proof of the next theorem.

\begin{enumerate}
\item There exists a $3$-diagonal magic square of order $n\geq 3$ with $n$ odd such that each diagonal consists of
$n$ consecutive numbers.

\item There exists a $4$-diagonal magic square of order $n\geq 4$ such that each diagonal consists of
$n$ consecutive numbers.

\item There exists a $5$-diagonal magic square of order $n\geq 5$ with $n$ is odd such that each diagonal consists
 of $n$ consecutive numbers.

\item There exists a $6$-diagonal magic square of order $n\geq 6$ such that there are four diagonals with numbers in $\{0,1,2,\ldots, 4n-1\}$ and each of these diagonals consists of $n$ consecutive numbers.

\item Let $3\leq s\leq m$ and either $s$ is even or $m$ is odd. Then there is an $MS(m;s)$
with $s-2$ diagonals which partition the set $\{0,1,\ldots, (m(s-2)-1\}$  and each of these diagonals consists of $m$ consecutive numbers.

\end{enumerate}
}
\end{remark}


In the following theorem we investigate the existence of a magic rectangle $MR(am,bm;bs,as)$, where $a+b=5$.
Note that this case is not covered by Corollary \ref{Cor:rectanglesmultiplication}.

\begin{theorem}\label{a+b=5}
Let $a,b$ be positive integers with $a+b=5$. Then there exists an $MR(am,bm;bs,as)$ if and only if $s$ is even and $2\leq s\leq m$.
\end{theorem}

\begin{proof}
If $a=1$ and $b=4$, then we are dealing with an $MR(m,4m;4s,s)$. By
Theorem \ref{Th.m.km.ks.s}, there is  an $MR(m,4m;4s,s)$ if and only if $s$ is even.

Now let $a=2$ and $b=3$. If there exists an $MR(2m,3m;3s,2s)$, then the row sum $(6ms-1)6ms/4m$ must be an integer, hence $s$ must be even.

Assume $s$ is even.
By Remark \ref{remark}, there is an $MS(2m;2s)$, say $S_1$,
with $s/2$ consecutive diagonals $d_0,d_1,\ldots, d_{(s-2)/2}$, which partition the set $\{0,1,2,\ldots, ms-1\}$  and each of these diagonals consists of $2m$ consecutive numbers.
Recall that the numbers in an $MS(2m;2s)$ are  $\{0,1,2,\ldots,4ms-1\}$.

If $s\geq 4$, we consider an $MR(m,2m; 2s,s)$, say $R$, which is constructed in the proof of Theorem \ref {Th.m.km.ks.s} using an  $MS(m;s)$, say $S_2$. If $s=2$, we consider an $MR(m,2m; 4,2)$ given in the proof of Theorem \ref{Th.m.km.2k.2}. Note that there are two diagonals in this $MR(m,2m; 4,2)$ which partition the set
$\{0,1,\ldots,2m-1\}$.

By definition, every number in
 $\{0,1,2,\ldots, ms-1\}$  appears exactly once in $S_2$ and every number in $\{0,1,2,\ldots, 2ms-1\}$ appears exactly once in $R$.
In addition, every diagonal in $S_2$ appears as a diagonal in $R$. See Figure \ref{6.4.3.6.4.2} for illustration.

Let $A$ be a $2m\times 2m$ square obtained from $S_1$ by adding $4ms$ to the diagonals $d_0,d_1,\ldots, d_{(s-2)/2}$ of $S_1$. Let $B$ be an $m\times 2m$ rectangle obtained from $R$ by adding $4ms$ to the diagonals of $R$ which consist of numbers in $\{ms,ms+1,ms+2,\ldots,2ms-1\}$.
Finally, let $C$ be a $2m\times 3m$ empty rectangle. We fill some of the cells of $C$ as follows:

\begin{enumerate}
\item For $0\leq i,j \leq 2m-1$, if $(i,j; k)\in A$, then
$(i,j; k)\in C$.

\item For $0\leq i\leq m-1$ and $0\leq j\leq 2m-1$, if $(i,j; k)\in B$, then $(j,i+2m; k)\in C$.
\end{enumerate}
It is easy to see that every number in the set $\{0, 1, 2,\ldots, 6ms-1\}$ appears exactly once in $C$.

The sum of each row of $C$ is:

$$\dfrac{4ms(4ms-1)}{4m}+ \dfrac{4ms^2}{2}+ \dfrac{2ms(2ms-1)}{4m}+ \dfrac{4ms^2}{2}=\dfrac{3s(6ms-1)}{2},$$
as desired.

The some of each column of $C$ for column $0\leq i\leq 2m-1$ is

$$\dfrac{4ms(4ms-1)}{4m}+ \dfrac{4ms^2}{2}=6ms^2-s,$$

and the column sum of $C$ for column $2m\leq i\leq 3m-1$ is
$$\dfrac{2ms(2ms-1)}{2m}+ 4ms^2=6ms^2-s,$$
as desired. See Figure \ref{6.9.6.4} for illustration.
\end{proof}

\begin{figure}[ht]
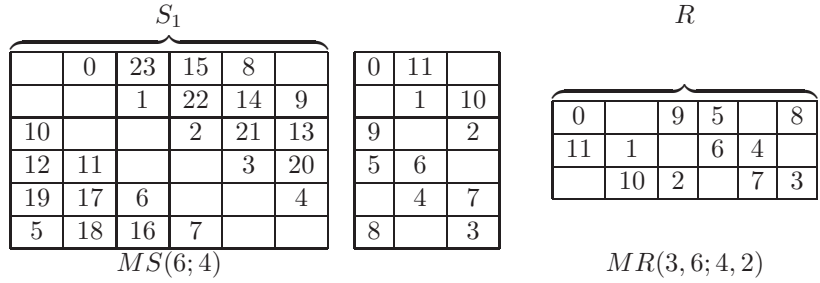

$$\begin{array}{cccc}
 S_1&&&R \\
\overbrace{
\begin{array}{|c|c|c|c|c|c|} \hline
&0&23&15&8&	\\ \hline
&&1&22&14&9   \\ \hline
10&&&2&21&13   \\ \hline
12&11&&&3&20 \\ \hline	
19&17&6	&&&4  \\ \hline
5&18&16&7&&  \\ \hline
\end{array}
}
&
\begin{array}{|c|c|c|} \hline
0&11 &	\\ \hline
&1&10   \\ \hline
9&&2   \\ \hline
5&6& \\ \hline	
&4&7  \\ \hline
8&&3  \\ \hline
\end{array}
&&
\overbrace{
\begin{array}{|c|c|c|c|c|c|} \hline
0&&9&5&&8  \\ \hline
11&1&&6&4& \\ \hline
&10&2&&7&3 \\ \hline
\end{array}
}\\
MS(6;4)&&&MR(3,6;4,2)
\end{array}
$$
\caption{$MS(6;4)$ and $MR(3,6;4,2)$}
 \label{6.4.3.6.4.2}
\end{figure}

\begin{figure}[ht]
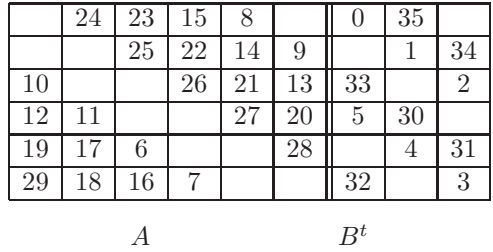

$$\begin{array}{ccc}
\begin{array}{|c|c|c|c|c|c||c|c|c|} \hline
&24&23&15&8&&0&35 &	\\ \hline
&&25&22&14&9&&1&34   \\ \hline
10&&&26&21&13&33&&2   \\ \hline
12&11&&&27&20&5&30& \\ \hline	
19&17&6	&&&28&&4&31  \\ \hline
29&18&16&7&&&32&&3  \\ \hline
\end{array}\vspace{3mm}\\
 A  \hspace{25mm} B^t
\end{array}$$
\caption{An $MR(6,9;6,4)$}
 \label{6.9.6.4}
\end{figure}

If there exists an $MR(m,n;r,s)$ with $\gcd(r,s)=1$, then $m=as$ and $n=ar$ for some positive integer $a$.

\begin{theorem}\label{am.bm.b.a}
Let $a, b, c$ be positive integers with $2 \leq a \leq b$. Let
$a, b, c$ be all odd, or $a$ and $b$ are
both even, $c$ is arbitrary, and $(a, b)\neq (2, 2)$.
Then there exists a magic rectangle set $MR(ac,bc;b,a)$.
\end{theorem}

\begin{proof}
By assumption and Theorem \ref {magicrectangleset} there is an $MRS(a,b;c)$, say $S_{k}$, where
$0\leq k\leq c-1$.
Partition an empty $ac\times bc$ rectangle, say $M$, into $c^2$ empty rectangles of size $a\times b$, say $T_{ij}$,
$0\leq i,j\leq c-1$. Replace the empty rectangle $T_{kk}$ with the rectangle $S_k$ for $0\leq k\leq c-1$. The resulting rectangle is an an $MR(ac,bc;b,a)$.
\end{proof}

\noindent{\bf Acknowledgement.}  The authors would like to thank the referee for his/her useful suggestions and comments.

\end{document}